\documentclass[twoside]{amsart}

\usepackage{amsmath,amssymb,amsfonts,amsthm,latexsym}
\usepackage{graphicx,xcolor}
\usepackage{enumitem}
\usepackage{subcaption}
\usepackage[margin=1.25in]{geometry}
\usepackage{hyperref}
\usepackage{ytableau}
\usepackage{mathdots}
\usepackage{tikz-cd}
\usepackage{orcidlink}

\newtheorem{theorem}{Theorem}[section]
\newtheorem{lemma}[theorem]{Lemma}
\newtheorem{proposition}[theorem]{Proposition}
\newtheorem{corollary}[theorem]{Corollary}

\newtheorem{mainthm}{Theorem}

\theoremstyle{definition}
\newtheorem{definition}[theorem]{Definition}
\newtheorem{notation}[theorem]{Notation}
\newtheorem{example}[theorem]{Example}
\theoremstyle{remark}
\newtheorem{remark}[theorem]{Remark}

\numberwithin{equation}{section}

\newcommand{\C}{\mathbb{C}}
\newcommand{\Sp}{\mathbb{S}}

\newcommand{\R}{\mathbb{R}}

\newcommand{\Z}{\mathbb{Z}}

\newcommand{\SU}{\operatorname{SU}}

\newcommand{\bs}{\boldsymbol}

\newcommand{\Hilb}{\operatorname{Hilb}}

\begin{document}

\title[Ratios of Schur polynomials and symplectic quotients]{The partial derivative of ratios of Schur polynomials and applications to symplectic quotients}

\author[H.-C.~Herbig]{Hans-Christian Herbig\,\orcidlink{0000-0003-2676-3340}}
\address{Departamento de Matem\'{a}tica Aplicada, Universidade Federal do Rio de Janeiro,
Av. Athos da Silveira Ramos 149, Centro de Tecnologia - Bloco C, CEP: 21941-909 - Rio de Janeiro, Brazil}

\author[D.~Herden]{Daniel Herden}
\address{Department of Mathematics,
Baylor University,
Sid Richardson Building,
1410 S.4th Street,
Waco, TX 76706, USA}
\email{daniel\_herden@baylor.edu}

\author[H.~Kolehmainen]{Harper Kolehmainen}
\address{Department of Mathematics and Statistics,
Rhodes College,
2000 N. Parkway,
Memphis, TN, 38112, USA}
\email{harperkole@gmail.com}

\author[C.~Seaton]{Christopher Seaton\,\orcidlink{0000-0001-8545-2599}}
\address{Department of Mathematics and Statistics,
Skidmore College,
815 North Broadway,
Saratoga Springs, NY 12866, USA}
\email{cseaton@skidmore.edu}


\dedicatory{Dedicated to the family of Hans-Christian Herbig in his memory.}

\thanks{
D.H. was supported by Simons Foundation grant MPS-TSM-00007788.
H.K. was supported by the Rhodes College Summer Fellowship Program.
C.S. was supported by an AMS-Simons Research Enhancement Grant for PUI Faculty.
}


\begin{abstract}
We show that a ratio of Schur polynomials $s_{\lambda}/s_{\rho}$ associated to partitions $\lambda$ and $\rho$ such that $\lambda\subsetneq\rho$ has a negative partial derivative at any point where all variables are positive. This is accomplished by establishing an injective map between sets of pairs of skew semistandard Young tableaux that preserves the product of the corresponding monomials. We use this result and the description of the first Laurent coefficient of the Hilbert series of the graded algebra of regular functions on a linear symplectic quotient by the circle to demonstrate that many such symplectic quotients are not graded regularly diffeomorphic. In addition, we give an upper bound for this Laurent coefficient in terms of the largest two weights of the circle representation and demonstrate that all but finitely many circle symplectic quotients of each dimension are not graded regularly diffeomorphic to linear symplectic quotients by $\operatorname{SU}_2$.
\end{abstract}

\maketitle

\tableofcontents


\section{Introduction}
\label{sec:intro}

Schur polynomials are an important class of symmetric polynomials indexed by integer partitions that arise in many contexts in representation theory
and combinatorics. The Schur polynomials admit several definitions, including a combinatorial description in which the terms of a Schur polynomial enumerate
semistandard Young tableaux (SSYT) of a given shape. Recall that if $\lambda$ is a partition, a SSYT of shape $\lambda$ is a Young diagram of shape
$\lambda$ whose boxes are labeled in such a way that the columns are increasing and the rows are non-decreasing.
See \cite[Chapter~7]{StanleyEC2}, \cite[Chapter~4]{Sagan}, or \cite[Chapter~1]{Macdonald} for background on Schur polynomials.

The first, second, and fourth authors have used Schur polynomials to describe the first Laurent coefficient of the Hilbert series of the graded algebra of
regular functions on a linear symplectic quotient by the circle $\Sp^1$ \cite{HerbigSeatonHSeries} or special unitary group $\SU_2$ \cite{HerbigHerdenSeatonSU2}.
In particular, the first Laurent coefficient in the case of symplectic quotients by $\Sp^1$ is given by a ratio of two Schur polynomials
where the diagram associated to the numerator is formed from that associated to the denominator by removing a single box.
This coefficient has been used to distinguish between symplectic quotients up to any equivalence that preserves the grading of the algebra
of regular functions. An experimental examination of this coefficient has indicated that it is a decreasing function of the weights.

The main goal of this paper is to establish this fact in the general setting of a ratio of Schur polynomials such that the numerator is associated
to a partition $\lambda$ and the denominator is associated to a partition $\rho$ with $\lambda\subsetneq\rho$. Our main result is the following.

\begin{mainthm}
\label{mainthm:PartDerDec}
Let $N \geq n \geq 1$ be fixed integers.
Let $\bs{x} = (x_1,x_2,\ldots,x_N)$ be a set of variables, and let $\lambda = (\lambda_1,\lambda_2,\ldots,\lambda_n)$
and $\rho = (\rho_1,\rho_2,\ldots,\rho_n)$ be partitions such that $\lambda_i \leq \rho_i$ for each $i$ and $\lambda_i < \rho_i$
for at least one $i$. Then for positive values of the $x_i$, $s_{\lambda}(\bs{x})/s_{\rho}(\bs{x})$ is strictly decreasing in
each variable. More strongly, for each $j$, the partial derivative
\begin{equation}
\label{eq:MainPartDeriv}
    \frac{\partial}{\partial x_j}\left(\frac{s_{\lambda}(\bs{x})}{s_{\rho}(\bs{x})}\right)
\end{equation}
can be expressed as the ratio of polynomials where the coefficients of the denominator are nonnegative and the coefficients of the
numerator are nonpositive.
\end{mainthm}

After the preparation of this manuscript, the authors learned that Theorem~\ref{mainthm:PartDerDec} has also been proven
by Khare-Tao \cite{KhareTao}. Using results about immanants in \cite{LamPostnikovPylyavskyy},
Khare-Tao consider the reciprocal $s_{\rho}(\bs{x})/s_{\lambda}(\bs{x})$ of the function
discussed here and show that, when the partial derivative is expressed with denominator $s_{\lambda}(\bs{x})^2$, each coefficient of
the numerator as a polynomial in $x_j$ is \emph{Schur-positive}, meaning that it is a sum of Schur polynomials.
Schur-positivity implies \emph{monomial-positivity}, meaning that the coefficient of each monomial is nonnegative, and this is equivalent to
Theorem~\ref{mainthm:PartDerDec}. Monomial-positivity then implies that the partial derivative is \emph{numerically positive},
meaning that its value for positive values of the $x_i$ is positive. See \cite[Proposition~8.1 and Theorem~8.6]{KhareTao}.
Numerical positivity was proven independently by Vishwakarma in \cite[Theorem~4.18 and Proposition~4.10]{Vishwakarma}. See also
\cite[Theorem~1.4]{BeltonGuillotKharePutinar} for a strengthening of numerical positivity and other applications. An alternate proof
of numerical positivity using
the theory of Chebyshev blossoming in M\"{u}ntz spaces was given by Rachid Ait-Haddou as noted after
\cite[Corollary~8.7]{KhareTao}.
Here, we offer a new self-contained proof using combinatorial and mostly elementary methods.

The proof of Theorem~\ref{mainthm:PartDerDec} is given in Section~\ref{sec:Proof}. In Section~\ref{subsec:ProofPartDerivAsSkew}, we give an expression for
the partial derivative of a ratio of Schur polynomials associated to arbitrary partitions in terms of skew Schur polynomials. Using this expression, we reduce
the proof of Theorem~\ref{mainthm:PartDerDec} to the existence of an injective map between sets of pairs of skew SSYT's. The remainder of
Section~\ref{sec:Proof} then establishes the existence of this map. In Section~\ref{sec:SympQuot}, we consider the question of distinguishing
between linear symplectic quotients using the first Laurent coefficient of the Hilbert series of the graded algebra of regular functions
and give an application of Theorem~\ref{mainthm:PartDerDec} as well as similar observations about ratios as in Equation~\eqref{eq:MainPartDeriv}.

Throughout this paper, we use the following notation, for the most part following \cite[Chapter~7]{StanleyEC2}.
An integer partition will be denoted $\rho = (\rho_1,\rho_2,\ldots,\rho_n)$
with $\rho_1\geq\rho_2\geq\cdots\geq\rho_n\geq 0$. The notation $\rho/\lambda$ indicates the skew shape formed by removing the shape
$\lambda$ from the shape $\rho$; if $d$ is an integer, then $\rho/d$ indicates the skew shape formed by removing the first $d$ boxes from the first row of $\rho$.
For convenience of notation, we do allow the $\rho_i$ to be equal to zero,
though the resulting partition and corresponding Schur polynomial can be considered identical to that with the zero entries removed.
A skew SSYT of shape $\rho/\lambda$ is a skew Young diagram of shape $\rho/\lambda$ labeled following the same requirements as for a SSYT.
If $T$ is a (skew) SSYT labeled by $\{1,2,\ldots,N\}$ and $\bs{x} = (x_1,x_2\ldots,x_N)$ is a collection of variables, then
$\bs{x}^T$ denotes the monic monomial formed by replacing labels $i$ in $T$ with variables $x_i$ and multiplying them together.
The entry of $T$ in position $(i,j)$
is denoted $T_{i,j}$. We will succinctly use the phrase \emph{the (skew) SSYT requirements} to refer to the requirements that
$T_{i,j}\leq T_{i+1,j}$ and $T_{i,j} < T_{i,j+1}$ whenever these values are defined.
Given a partition $\rho = (\rho_1, \rho_2, \ldots,\rho_n)$ and a collection of variables
$\bs{x} = (x_1,\ldots, x_N)$, the Schur polynomial $s_\rho(\bs{x})$ can be defined
as
\[
    s_\rho(\bs{x})
    =
    \sum\limits_{T} \bs{x}^T,
\]
where the sum is over all SSYT's of shape $\rho$. The skew Schur polynomial $s_{\rho/\lambda}$ is defined identically where
the sum is over all skew SSYT's of shape $\rho/\lambda$.
We will sometimes identify a partition with the set of positions in the corresponding diagram; i.e., we will say $(i,j)\in\rho/\lambda$
to indicate that the diagram $\rho/\lambda$ contains a box in position $(i,j)$.


\section*{Acknowledgements}

We would like to thank Apoorva Khare for pointing out several related references and past proofs of Theorem~\ref{mainthm:PartDerDec}
as well as the anonymous referees for very helpful comments that improved this manuscript.
This paper developed from H.K.'s Summer Research Fellowship project and continued research in
the Rhodes College Department of Mathematics and Statistics, and the
authors gratefully acknowledge the support of the department and college for these
activities.
D.H. was supported by Simons Foundation grant MPS-TSM-00007788.
H.K. was supported by the Rhodes College Summer Fellowship Program.
C.S. was supported by an AMS-Simons Research Enhancement Grant for PUI Faculty.


\section{Proof of Theorem~\ref{mainthm:PartDerDec}}
\label{sec:Proof}

Throughout this section, fix a finite set of variables $\bs{x} = (x_1,x_2,\ldots,x_N)$ and let $\bs{x}_{N-1} = (x_1,x_2,\ldots,x_{N-1})$.
We will always assume $N \geq n$ where $\lambda = (\lambda_1,\ldots,\lambda_n)$ and $\rho = (\rho_1,\ldots,\rho_n)$.
Recall that Theorem~\ref{mainthm:PartDerDec} considers the partial derivative with respect to a variable $x_j$. It is sufficient by symmetry
to consider the case $j = N$, so we will restrict to this case throughout this section.

To begin, let us briefly outline the strategy of the proof. In Section~\ref{subsec:ProofPartDerivAsSkew}, we give expression~\eqref{eq:PartDerivAsSkew}
for the partial derivative of the ratio of two arbitrary Schur polynomials in terms of skew Schur polynomials.
From this expression, it is clear that Theorem~\ref{mainthm:PartDerDec} follows if each of the terms corresponding to $e > d$ is cancelled
by terms corresponding to values of $e < d$. In Section~\ref{subsec:ProofReductionInject}, we establish notation to explain how this cancellation
will be demonstrated, and in Section~\ref{subsec:ProofSimpleCase} detail a simple specific case that illustrates the idea behind the proof. Section~\ref{subsec:ProofInjection} contains most of the technical constructions and completes the proof in the case that $\lambda$ and $\rho$
differ by a single box, and Section~\ref{subsec:ProofGeneralCase} extends this to the general case.


\subsection{The partial derivative in terms of skew Schur polynomials}
\label{subsec:ProofPartDerivAsSkew}


Our first task is to develop an expression for the partial derivative \eqref{eq:MainPartDeriv} with $j = N$.
See \cite{GrinbergKorniichukEA} for another recent result about partial derivatives of Schur polynomials.
We begin with the following, which benefited from \cite{MvLmathoverflow}.
As usual, $s_{\rho/d}$ denotes the skew Schur polynomial formed by removing a row of size $d$ from $\rho$.

\begin{lemma}
\label{lem:SchurAsSkew}
Let $\rho = (\rho_1,\rho_2,\ldots,\rho_n)$ be a partition. Then
\begin{equation}
\label{eq:SchurAsSkew}
    s_\rho(\bs{x})
    =
    \sum\limits_{d=0}^{\rho_1} x_N^d s_{\rho/d}(\bs{x}_{N-1}).
\end{equation}
\end{lemma}
\begin{proof}
In a SSYT with entries at most $N$, the label $N$ can only occur in the last entry of a column.
There are $\rho_1$ columns in the shape $\rho$ and therefore at most $\rho_1$ boxes with entry $N$. For each
$d = 0, 1, \ldots, \rho_1$, let $T_1^{(d)},\ldots, T_{k_d}^{(d)}$ denote the list of SSYT's of shape $\rho$
in which $N$ occurs $d$ times and let $\bs{x}_{N-1}^{T_i^{(d)}\ast}$ denote the monomial $\bs{x}^{T_i^{(d)}}$
with the $x_N$ factors removed so that $\bs{x}^{T_i^{(d)}} = x_N^d \bs{x}_{N-1}^{T_i^{(d)}\ast}$.
Then we can express
\[
    s_\rho(\bs{x})
    =
    \sum\limits_{d=0}^{\rho_1} x_N^d\sum\limits_{i=1}^{k_d} \bs{x}_{N-1}^{T_i^{(d)}\ast}.
\]
Recall \cite[p.~339]{StanleyEC2} that a \emph{horizontal strip of length $d$} is defined to be a skew shape $\rho/\nu$
with no two squares in the same column. Let $\rho - d$ denote the set of Young diagrams formed from $\rho$ by removing
a horizontal strip of length $d$. Then $\bs{x}_{N-1}^{T_i^{(d)}\ast}$, $i = 1,\ldots, k_d$ is a list of all monomials
corresponding to a SSYT whose shape is an element of $\rho - d$. That is, we have
\begin{equation}
\label{eq:SchurAs-d}
    s_\rho(\bs{x})
    =
    \sum\limits_{d=0}^{\rho_1} x_N^d\sum\limits_{\mu\in\rho - d} s_\mu(\bs{x}_{N-1}).
\end{equation}
By Pieri's rule \cite[Corollary~7.15.9]{StanleyEC2},
\[
    \sum\limits_{\mu\in\rho - d} s_\mu(\bs{x}_{N-1})
    =
    s_{\rho/d}(\bs{x}_{N-1})
\]
which applied to Equation~\eqref{eq:SchurAs-d} completes the proof.
\end{proof}

We now give the following expression for the partial derivative of a ratio of Schur polynomials in terms of skew Schur polynomials.
Equation~\eqref{eq:PartDerivAsSkew} below also appears in the proof of \cite[Theorem~8.6]{KhareTao}, where it was proven using the same argument.

\begin{corollary}
\label{cor:PartDerivAsSkew}
Let $\lambda = (\lambda_1,\lambda_2,\ldots,\lambda_n)$ and $\rho = (\rho_1,\rho_2,\ldots,\rho_n)$ be partitions. Then
\begin{equation}
\label{eq:PartDerivAsSkew}
    \frac{\partial}{\partial x_N}\left(\frac{s_{\lambda}(\bs{x})}{s_\rho(\bs{x})}\right)
    =
    \frac{
        \sum\limits_{d=0}^{\rho_1} \sum\limits_{e=0}^{\lambda_1}
            (e - d) x_N^{d + e - 1} s_{\rho/d}(\bs{x}_{N-1}) s_{\lambda/e}(\bs{x}_{N-1})
        }
        {s_\rho(\bs{x})^2}.
\end{equation}
\end{corollary}
\begin{proof}
Using Lemma~\ref{lem:SchurAsSkew} and the quotient rule, we compute the partial derivative in Equation~\eqref{eq:PartDerivAsSkew} to be
\begin{align*}
    &
    \frac{s_\rho(\bs{x}) \frac{\partial s_{\lambda}(\bs{x})}{\partial x_N}
        - s_\lambda(\bs{x}) \frac{\partial s_\rho(\bs{x})}{\partial x_N} }
        {s_\rho(\bs{x})^2}
    \\&\quad=
    \frac{
        \sum\limits_{d=0}^{\rho_1} x_N^d s_{\rho/d}(\bs{x}_{N-1})
        \frac{\partial }{\partial x_N}
        \sum\limits_{e=0}^{\lambda_1} x_N^e s_{\lambda/e}(\bs{x}_{N-1})
        -
        \sum\limits_{e=0}^{\lambda_1} x_N^e s_{\lambda/e}(\bs{x}_{N-1})
        \frac{\partial }{\partial x_N}
        \sum\limits_{d=0}^{\rho_1} x_N^d s_{\rho/d}(\bs{x}_{N-1})
        }{s_\rho(\bs{x})^2}
    \\&\quad=
    \frac{
        \sum\limits_{d=0}^{\rho_1} \sum\limits_{e=0}^{\lambda_1}
            e x_N^{d + e - 1} s_{\rho/d}(\bs{x}_{N-1}) s_{\lambda/e}(\bs{x}_{N-1})
        -
        \sum\limits_{e=0}^{\lambda_1} \sum\limits_{d=0}^{\rho_1}
        d x_N^{d + e - 1} s_{\rho/d}(\bs{x}_{N-1}) s_{\lambda/e}(\bs{x}_{N-1})
        }{s_\rho(\bs{x})^2}
    \\&\quad=
    \frac{
        \sum\limits_{d=0}^{\rho_1} \sum\limits_{e=0}^{\lambda_1}
            (e - d) x_N^{d + e - 1} s_{\rho/d}(\bs{x}_{N-1}) s_{\lambda/e}(\bs{x}_{N-1})
        }{s_\rho(\bs{x})^2}.
        \qedhere
\end{align*}
\end{proof}


\subsection{Reduction to the existence of an injective map}
\label{subsec:ProofReductionInject}

We now assume that $\lambda = (\lambda_1,\lambda_2,\ldots,\lambda_n)$ and $\rho = (\rho_1,\rho_2,\ldots,\rho_n)$ are partitions such
that $\lambda\subsetneq\rho$, i.e., $\lambda_i \leq \rho_i$ for each $i$ and $\lambda_i < \rho_i$ for at least one $i$.
To prove Theorem~\ref{mainthm:PartDerDec}, we will show that the numerator of the right side of Equation~\eqref{eq:PartDerivAsSkew}
consists of terms with nonpositive coefficients, i.e., that the terms with positive coefficients all cancel. Note that each positive
term is of the form $\bs{x}_{N-1}^{S} \bs{x}_{N-1}^{T}$ where $S$ is a skew SSYT of shape $\rho/d$,
$T$ is a skew SSYT of shape $\lambda/e$, and $e > d$.

\begin{notation}
\label{not:S(d,e)P(d,e)}
For $0\leq d\leq\rho_1$ and $0\leq e\leq\lambda_1$, let $\bs{S}_{(d,e)}$ denote the set of pairs $(S, T)$ where $S$ is a skew SSYT of
shape $\rho/d$ and $T$ is a skew SSYT of shape $\lambda/e$, both with entries at most $N - 1$. For $(S,T)\in\bs{S}_{(d,e)}$, or more generally
for a pair of tableaux with labels $\leq N - 1$, we let
\[
    \bs{x}_{N-1}^{(S,T)} = \bs{x}_{N-1}^S \bs{x}_{N-1}^T.
\]
Further, let
\begin{equation}
\label{eq:defP(d,e)}
    \bs{P}(d,e) = (e - d) x_N^{d + e - 1} s_{\rho/d}(\bs{x}_{N-1}) s_{\lambda/e}(\bs{x}_{N-1}),
\end{equation}
so that the numerator of the right side of Equation~\eqref{eq:PartDerivAsSkew} can be written
\[
    \sum\limits_{d=0}^{\rho_1} \sum\limits_{e=0}^{\lambda_1} \bs{P}(d,e).
\]
\end{notation}

To prove Theorem~\ref{mainthm:PartDerDec}, it is sufficient to construct for each $e > d$ an injective map
\[
    \varphi\colon\bs{S}_{(d,e)}\to\bs{S}_{(e,d)}
\]
such that
\[
    \bs{x}_{N-1}^{(S,T)} = \bs{x}_{N-1}^{\varphi(S,T)}.
\]
It would follow that when $e \neq d$, $\bs{P}(d,e) + \bs{P}(e,d)$ consists of nonpositive terms.


\subsection{A simple case}
\label{subsec:ProofSimpleCase}

Before considering the construction of the injective map $\varphi$  in the next section, let us first consider a special case that arose in the applications
considered in Section~\ref{sec:SympQuot}. Here, we will define a map $\psi$ that plays the role of $\varphi$.
Along with serving to illustrate the idea behind the proof, it is interesting that in this case, $\psi$ turns out to be a bijection.
Note that this result does not hold if $N > n + 1$. A simple computation verifies that in the case of $\rho = (2,1)$ and $\lambda = (1,1)$, if $N = 3$,
then $\bs{S}_{(0,1)}$ and $\bs{S}_{(1,0)}$ contain a different number of elements.

\begin{lemma}
\label{lem:ProofSimpleCase}
Let $\lambda = (\lambda_1,\lambda_2,\ldots,\lambda_n)$ and $\rho = (\rho_1,\rho_2,\ldots,\rho_n)$ be partitions such
that $\lambda_i \leq \rho_i$ for each $i$ and $\lambda_i < \rho_i$ for at least one $i$. Assume that the $\lambda_i$
and $\rho_i$ are all nonzero and the number of variables $N = n + 1$.
There is a bijection $\psi\colon\bs{S}_{(0,1)}\to\bs{S}_{(1,0)}$ such that $\bs{x}_{N-1}^{(S,T)} = \bs{x}_{N-1}^{\psi(S,T)}$,
whereby $\bs{P}(1,0) + \bs{P}(0,1) = 0$.
\end{lemma}
\begin{proof}
Let $(S, T)\in\bs{S}_{(0,1)}$. Then $S$ is a SSYT of shape $\rho$ and $T$ a skew SSYT of shape $\lambda/1$, both with entries at most
$N - 1 = n$. Because the first column of $S$ is increasing and has $n$ entries, this column is determined; we have $S_{i,1} = i$ for each $i$,
and the same holds for any SSYT of shape $\lambda$. Then as the first column of $T$ is increasing and $T_{n,1}\leq n$, we have
$T_{i,1}\leq i$ for each $i$, and the same holds for any skew SSYT of shape $\rho/1$. As the second column of $S$ is
increasing, we have $S_{i,2}\geq i$ for each $i$; similarly, $T_{i,2}\geq i$ for each $i$, and the same again holds for any (skew) SSYT
of shape $\lambda$ or $\lambda/1$.

Based on these above observations, we can define $\psi(S,T) = (S^\prime, T^\prime)\in\bs{S}_{(1,0)}$ by swapping the first columns of
$S$ and $T$, i.e., $S^\prime$ is formed from $S$ by replacing the first column with that of $T$, and $T^\prime$ is formed from $T$
by replacing the first column with that of $S$. See Figure~\ref{fig:GenDecreas1001}.
This process is clearly invertible, implying that $\psi\colon\bs{S}_{(0,1)}\to\bs{S}_{(1,0)}$ is a bijection
that preserves the products of the corresponding monomials as required.
\end{proof}

\begin{figure}[!htb]
\centering
\begin{subfigure}{.4\textwidth}
\centering
\ytableausetup{mathmode, boxframe=normal, boxsize=2em}
\begin{ytableau}
\scriptstyle 1  &   \scriptstyle S_{1,2}    &   \scriptstyle S_{1,3}    &   \none[\scriptstyle\cdots]   &   \none[\scriptstyle\cdots]
                &   \scriptstyle S_{1,\rho_1}
\\
\scriptstyle 2 &    \scriptstyle S_{2,2}    &   \scriptstyle S_{2,3}    &   \none[\scriptstyle\cdots]   &   \scriptstyle S_{2,\rho_2}
\\
\scriptstyle 3 &    \scriptstyle S_{3,2}    &   \none[\scriptstyle\cdots]   &   \scriptstyle S_{2,\rho_3}
\\
\none[\scriptstyle\vdots]   &   \none[\scriptstyle\vdots]   &   \none[\scriptstyle\qquad\iddots]
\\
\scriptsize n   &   \none[\scriptstyle\cdots] &   \scriptstyle S_{n,\rho_n}
\end{ytableau}
\caption*{$S$}
\end{subfigure}
\begin{subfigure}{.4\textwidth}
\centering
\ytableausetup{mathmode, boxframe=normal, boxsize=2em}
\begin{ytableau}
\none   &   *(lightgray) \scriptstyle T_{1,2}   &   *(lightgray) \scriptstyle T_{1,3}   & \none[\scriptstyle\cdots] &   \none[\scriptstyle\cdots]
        &   *(lightgray) \scriptstyle T_{1,\lambda_1}
\\
*(lightgray) \scriptstyle T_{2,1}   &   *(lightgray) \scriptstyle T_{2,2}   &   *(lightgray) \scriptstyle T_{2,3}   &   \none[\scriptstyle\cdots]
        &   *(lightgray) \scriptstyle T_{2,\lambda_2}
\\
*(lightgray) \scriptstyle T_{3,1}   &   *(lightgray) \scriptstyle T_{3,2}   &   \none[\scriptstyle\cdots]
        &   *(lightgray) \scriptstyle T_{3,\lambda_3}
\\
\none[\scriptstyle\vdots]   &   \none[\scriptstyle\vdots]   &   \none[\scriptstyle\qquad\iddots]
\\
*(lightgray) \scriptstyle T_{n,1}   &   \none[\scriptstyle\cdots]   &   *(lightgray) \scriptstyle T_{n,\lambda_n}
\end{ytableau}
\caption*{$T$}
\end{subfigure}

\bigskip

\begin{subfigure}{\textwidth}
\centering
\begin{tikzcd}
(S, T)\in\bs{S}_{(0,1)} \arrow[mapsto]{d}{\psi}
\\
(S^\prime, T^\prime)\in\bs{S}_{(1,0)}
\end{tikzcd}
\end{subfigure}

\bigskip

\begin{subfigure}{.4\textwidth}
\centering
\ytableausetup{mathmode, boxframe=normal, boxsize=2em}
\begin{ytableau}
\none           &   \scriptstyle S_{1,2}    &   \scriptstyle S_{1,3}    &   \none[\scriptstyle\cdots]   &   \none[\scriptstyle\cdots]
                &   \scriptstyle S_{1,\rho_1}
\\
*(lightgray) \scriptstyle T_{2,1}
                &    \scriptstyle S_{2,2}    &   \scriptstyle S_{2,3}    &   \none[\scriptstyle\cdots]   &   \scriptstyle S_{2,\rho_2}
\\
*(lightgray) \scriptstyle T_{3,1}
                &    \scriptstyle S_{3,2}    &   \none[\scriptstyle\cdots]   &   \scriptstyle S_{3,\rho_3}
\\
\none[\scriptstyle\vdots]   &   \none[\scriptstyle\vdots]   &   \none[\scriptstyle\qquad\iddots]
\\
*(lightgray) \scriptstyle T_{n,1}
                            &   \none[\scriptstyle\cdots]   &   \scriptstyle S_{n,\rho_n}
\end{ytableau}
\caption*{$S^\prime$}
\end{subfigure}
\begin{subfigure}{.4\textwidth}
\centering
\ytableausetup{mathmode, boxframe=normal, boxsize=2em}
\begin{ytableau}
\scriptstyle 1  &   *(lightgray) \scriptstyle T_{1,2}   &   *(lightgray) \scriptstyle T_{1,3}   & \none[\scriptstyle\cdots] &   \none[\scriptstyle\cdots]
        &   *(lightgray) \scriptstyle T_{1,\lambda_1}
\\
\scriptstyle 2   &   *(lightgray) \scriptstyle T_{2,2}   &   *(lightgray) \scriptstyle T_{2,3}   &   \none[\scriptstyle\cdots]
        &   *(lightgray) \scriptstyle T_{2,\lambda_2}
\\
\scriptstyle 3   &   *(lightgray) \scriptstyle T_{3,2}   &   \none[\scriptstyle\cdots]
        &   *(lightgray) \scriptstyle T_{3,\lambda_3}
\\
\none[\scriptstyle\vdots]   &   \none[\scriptstyle\vdots]   &   \none[\scriptstyle\qquad\iddots]
\\
\scriptsize n   &   \none[\scriptstyle\cdots] &   *(lightgray) \scriptstyle T_{n,\lambda_n}
\end{ytableau}
\caption*{$T^\prime$}
\end{subfigure}
\caption{Illustration of the bijection $\bs{S}_{(0,1)}\to\bs{S}_{(1,0)}$.}
\label{fig:GenDecreas1001}
\end{figure}

In the next section, we will define a similar map $\varphi$ for general values of $d < e$
under the assumption that $\lambda$ differs from $\rho$ by a single box, which we will see in Section~\ref{subsec:ProofGeneralCase}
is sufficient to prove Theorem~\ref{mainthm:PartDerDec}.
Though the construction is significantly more technical, the idea remains the same as that of Lemma~\ref{lem:ProofSimpleCase};
$\varphi(S,T)$ is formed by swapping a certain collection of labeled boxes between the two (skew) SSYT's.


\subsection{An injective map in the case that $\lambda$ and $\rho$ differ by one box}
\label{subsec:ProofInjection}

In this section, we construct the map $\varphi\colon\bs{S}_{(d,e)}\to\bs{S}_{(e,d)}$ described in Section~\ref{subsec:ProofReductionInject} in the case that
$\lambda$ is given by $\rho$ with a single box removed and hence prove Theorem~\ref{mainthm:PartDerDec} in this case; see Theorem~\ref{thrm:PartDerDecOneBox}.
\textbf{Throughout this section, we assume that there is exactly one $r$ such that $\lambda_i = \rho_i$ for $i\neq r$ and $\lambda_r = \rho_r - 1$}.
With this assumption on $\lambda$ and $\rho$, let us establish the following.

\begin{notation}
\label{not:defS_U}
Let $e > d$ and let $(S,T)\in\bs{S}_{(d,e)}$ so that $S$ and $T$ are skew SSYT's of
shapes $\rho/d$ and $\lambda/e$, respectively, with entries at most $N - 1$.
It will be convenient to adopt the convention that $T_{1,j} = 0$ for $d < j\leq e$.

Given any subset $U$ of positions in the skew shape $\rho/d$ such that $(r, \rho_r)\in U$, we define a pair of skew tableaux
$(S_U^\prime, T_U^\prime)$ by swapping the entries of $U$ between $S$ and $T$ and then interchanging the resulting tableaux. That is,
$S_U^\prime$ results from $T$ by adding a box at position $(r, \rho_r)$ with entry $S_{r, \rho_r}$ and replacing each $T_{i,j}$
such that $(i,j)\in U\smallsetminus \{ (r, \rho_r)\}$ with $S_{i,j}$, and $T_U^\prime$ results from $S$ by removing the box at
position $(r, \rho_r)$ and replacing each $S_{i,j}$ such that $(i,j)\in U\smallsetminus \{ (r, \rho_r)\}$ with $T_{i,j}$; see Figure~\ref{fig:FormingSUTU}.
Note that the pair $(S_U^\prime, T_U^\prime)$ depends on $S$, $T$, and $U$, but we will often refer to $S_U^\prime$ and $T_U^\prime$
individually when there is no risk of confusion.
Note further that $\bs{x}_{N-1}^{(S,T)} = \bs{x}_{N-1}^{(S_U^\prime,T_U^\prime)}$
as the labels are simply permuted between the two skew SSYT's $S$ and $T$.

We continue to use the convention that a box labeled $0$ in $S_U^\prime$ or $T_U^\prime$ is missing from the skew shape.
The shapes of $S_U^\prime$ or $T_U^\prime$ need not be skew shapes, and if they are,
the resulting $S_U^\prime$ or $T_U^\prime$ need not be skew SSYT's as they may fail the skew SSYT requirements. Moreover,
they are only tableaux of shape $\lambda/d$ and $\rho/e$, respectively, if $U$ is a subset of
the skew shape $\rho/e$; if this condition fails, then the formation of $(S_U^\prime, T_U^\prime)$ may involve moving $S_{1,j}$
with $d< j \leq e$ to $T$.
\end{notation}

\begin{figure}[!htb]
\centering
\begin{subfigure}{.4\textwidth}
\centering
\ytableausetup{mathmode, boxframe=normal, boxsize=1.7em}
\begin{ytableau}
\none &  \none   &     1  &   *(lightgray)  1  &   *(lightgray)  2
\\
 1  &   *(lightgray) 1   &     *(lightgray) 2
\\
 3
\end{ytableau}
\caption*{$S$}
\end{subfigure}
\vspace{.5cm}
\begin{subfigure}{.4\textwidth}
\centering
\ytableausetup{mathmode, boxframe=normal, boxsize=1.7em}
\begin{ytableau}
\none &  \none   &     \none  &   3
\\
 2  &   3   &     4
\\
 4
\end{ytableau}
\caption*{$T$}
\end{subfigure}

\begin{subfigure}{.4\textwidth}
\centering
\ytableausetup{mathmode, boxframe=normal, boxsize=1.7em}
\begin{ytableau}
\none &  \none   &     \none  &   *(lightgray) 1 &  *(lightgray) 2
\\
 2  &   *(lightgray) 1   &     *(lightgray) 2
\\
 4
\end{ytableau}
\caption*{$S_U^\prime$}
\end{subfigure}
\begin{subfigure}{.4\textwidth}
\centering
\ytableausetup{mathmode, boxframe=normal, boxsize=1.7em}
\begin{ytableau}
\none &  \none   &     1  &   *(lightgray) 3
\\
 1  &   *(lightgray) 3   &     *(lightgray) 4
\\
 3
\end{ytableau}
\caption*{$T_U^\prime$}
\end{subfigure}
\caption{An example of forming $(S_U^\prime, T_U^\prime)$ from $(S, T)$ where $U$ is the set of shaded positions in $S$.
Each box in $U$ is replaced by the box in the same position in $T$, or removed if $T$ contains no box in this position,
to form $T_U^\prime$. Similarly, each box in $U$ is placed in the same position in $T$ to form $S_U^\prime$.
The boxes that have been moved between $S$ and $T_U^\prime$, and $T$ and $S_U^\prime$, respectively, are shaded in $S_U^\prime$ and $T_U^\prime$.}
\label{fig:FormingSUTU}
\end{figure}

In order to define $\varphi$, we will define for each $(S,T)\in\bs{S}_{(d,e)}$ a specific set $U = U(S,T)\subset \rho/d$ such that
$(S_U^\prime, T_U^\prime)\in\bs{S}_{(e,d)}$, and in particular is a pair of skew SSYT's.
We will form $U$ recursively as follows. As the initial condition, let $U = \{(r,\rho_r)\}$.
If $(S_U^\prime, T_U^\prime)\notin\bs{S}_{(e,d)}$, i.e., $(S_U^\prime, T_U^\prime)$ is not a pair of skew SSYT's,
then one or both of these tableaux fail the skew SSYT requirements at the boundary of $U$.
That is, there is in either $S_U^\prime$ or $T_{U}^\prime$ (or both) an $(i,j)\notin U$ such that
the skew SSYT requirements fail for one of the pairs $(i,j), (i\pm 1,j)$ or $(i,j), (i,j\pm 1)$. Noting that there may be many such $(i,j)$, we
choose one and add it to $U$. The process terminates when $S_U^\prime$ and
$T_U^\prime$ are both skew SSYT's, i.e., there is no $(i,j)$ at which the skew SSYT requirements fail.

Of course, the description above is not well-defined; there can be many choices of $(i,j)$ at each step in the recursion,
and it is unclear whether the result depends on the choices made. Moreover, while a $U$ such that
$(S_U^\prime, T_U^\prime)$ is a pair of skew SSYT's will clearly eventually be obtained,
it is not clear that the resulting $U$ will be a subset of the skew shape $\rho/e$,
so the result $(S_U^\prime, T_U^\prime)$ may not be an element of $\bs{S}_{(e,d)}$. To address these concerns, let us make the following.

\begin{definition}[Spanning tree]
\label{def:SpanningTree}
Let $(S,T)\in\bs{S}_{(d,e)}$ and let $\{(B_k, C_k)\}_{1\leq k\leq M}$ be a sequence of pairs of positions
$B_k, C_k$ in the skew shape $\rho/d$ for some integer $M$.
Let $B_0 = (r, \rho_r)$, and for $K=0,\ldots,M$, let $U_K = \{B_0, B_1, \ldots, B_K\}$.
If the sequence $\{(B_k, C_k)\}_{1\leq k\leq M}$ needs to be specified, we denote
\[
    U_K\big(\{(B_k, C_k)\}_{1\leq k\leq M}\big) = \{B_0,B_1,\ldots,B_K\}.
\]

We say that $\{(B_k, C_k)\}_{1\leq k\leq M}$ is a \emph{spanning tree for $(S,T)$} if it satisfies the following requirements:
\begin{enumerate}
\item[(1)]  $B_k \in \rho/d \smallsetminus U_{k-1}$ and $C_k \in U_{k-1}$
            for each $k=1,\ldots,M$.
\item[(2)]  $B_k$ is adjacent to $C_k$ for $k=1,\ldots, M$, i.e., $C_k = B_k \pm (1,0)$ or $C_k = B_k \pm (0,1)$.
\item[(3)]  For $k=1,\ldots,M$, $S_{U_{k-1}}^\prime$ and $T_{U_{k-1}}^\prime$ are not
            both skew SSYT's, and $(B_k, C_k)$ provides a pair of adjacent positions at which one or both of
            $S_{U_{k-1}}^\prime$ and $T_{U_{k-1}}^\prime$ fail the skew SSYT requirements.
\item[(4)]  $\{(B_k, C_k)\}_{1\leq k\leq M}$ is maximal; i.e., there is no pair $(B_{M+1}, C_{M+1})$ of adjacent positions
            at which either $S_{U_M}^\prime$ or $T_{U_M}^\prime$ fail the skew SSYT requirements.
\end{enumerate}
We set $U = U\big(\{(B_k, C_k)\}_{1\leq k\leq M}\big) = U_M\big(\{(B_k, C_k)\}_{1\leq k\leq M}\big) = \{B_0,B_1,\ldots,B_M\}$.
\end{definition}

Note that Condition~(4) implies that $S_U^\prime$ and $T_U^\prime$ are skew SSYT's.
Note further that it is possible that $\emptyset$ is a spanning tree; this occurs when
$S_{\{(r,\rho_r)\}}^\prime$ and $T_{\{(r,\rho_r)\}}^\prime$ are skew SSYT's, where
$S_{\{(r,\rho_r)\}}^\prime$ is defined to be $T$ with a box added at position $(r,\rho_r)$ with label $S_{r,\rho_r}$ and
$T_{\{(r,\rho_r)\}}^\prime$ is defined to be $S$ with the box at position $(r,\rho_r)$ removed.

\begin{lemma}
\label{lem:UIndepSpanTree}
Let $(S,T)\in\bs{S}_{(d,e)}$.
The set $U\big(\{(B_k, C_k)\}_{1\leq k\leq M}\big) = \{B_0,B_1,\ldots,B_M\}$ is independent of the choice of spanning tree
$\{(B_k, C_k)\}_{1\leq k\leq M}$.
\end{lemma}
\begin{proof}
Fix $(S,T)\in\bs{S}_{(d,e)}$, let $\{(B_k, C_k)\}_{1\leq k\leq M}$ and $\{(B_k^\prime, C_k^\prime)\}_{1\leq k\leq M^\prime}$ be two choices of
spanning trees for $(S,T)$, and let $U = U\big(\{(B_k, C_k)\}_{1\leq k\leq M}\big)$ and
$U^\prime = U\big(\{(B_k^\prime, C_k^\prime)\}_{1\leq k\leq M^\prime}\big)$.
Assume $U\nsubseteq U^\prime$, and let $k_0$ be the smallest value of $k$ such that $B_k\notin U^\prime$. Then
$C_{k_0}\in U^\prime$ so that by Definition~\ref{def:SpanningTree}~(2) and (3), $(B_{k_0}, C_{k_0})$ is a pair of adjacent positions for which
one or both of $S_{U^\prime}^\prime$ and $T_{U^\prime}^\prime$ fail the skew SSYT requirements. It follows that
$\{(B_k^\prime, C_k^\prime)\}_{1\leq k\leq M^\prime}\cup\{(B_{k_0}, C_{k_0})\}$ satisfies Conditions~(1), (2), and (3), contradicting
Definition~\ref{def:SpanningTree}~(4) for $\{(B_k^\prime, C_k^\prime)\}_{1\leq k\leq M^\prime}$. Hence $U\subseteq U^\prime$,
and mutatis mutandis, $U^\prime\subseteq U$.
\end{proof}

\begin{notation}
\label{not:U(S,T)}
By Lemma~\ref{lem:UIndepSpanTree}, $U\big(\{(B_k, C_k)\}_{1\leq k\leq M}\big)$ is independent of the choice of spanning tree and depends only
on the pair $(S,T)\in\bs{S}_{(d,e)}$. Hence, we will hereafter denote this set $U(S,T)$.
\end{notation}

We now indicate that the set $U(S,T)$ is minimal among subsets $U^\ast$ of positions in $\rho/d$ that contain $(r,\rho_r)$ and such that
$S_{U^\ast}^\prime$ and $T_{U^\ast}^\prime$ are skew SSYT's.

\begin{lemma}
\label{lem:UMinimal}
Let $(S,T)\in\bs{S}_{(d,e)}$, and
let $U^\ast\subseteq\rho/d$ such that $(r,\rho_r)\in U^\ast$, and such that both $S_{U^\ast}^\prime$ and $T_{U^\ast}^\prime$ are skew SSYT's.
Then $U(S,T)\subseteq U^\ast$.
\end{lemma}
\begin{proof}
Fix $(S,T)\in\bs{S}_{(d,e)}$ as well as a spanning tree $\{(B_k, C_k)\}_{1\leq k\leq M}$ for $(S, T)$, and then
$U(S,T) = U\big(\{(B_k, C_k)\}_{1\leq k\leq M}\big)$ by Lemma~\ref{lem:UIndepSpanTree}. Assume for contradiction that
$U(S,T)\nsubseteq U^\ast$, and let $k_0$ be the smallest value of $k$ such that $B_k\notin U^\ast$. Then
$C_{k_0}\in U^\ast$ so that by Definition~\ref{def:SpanningTree}~(2) and (3), $(B_{k_0}, C_{k_0})$ is a pair of adjacent positions for which
one or both of $S_{U^\ast}^\prime$ and $T_{U^\ast}^\prime$ fail the skew SSYT requirements. Then either $S_{U^\ast}^\prime$ or $T_{U^\ast}^\prime$
is not a skew SSYT, a contradiction.
\end{proof}

We now claim that $S_{U(S,T)}^\prime$ has shape $\rho/e$ and $T_{U(S,T)}^\prime$ has shape $\lambda/d$, i.e., the formation of
$(S_{U(S,T)}^\prime,T_{U(S,T)}^\prime)$ does not involve moving $S_{1,j}$ with $d< j\leq e$ from $S$ to $T$.

\begin{lemma}
\label{lem:UinT}
We have $U(S,T)\subseteq \rho/e$.
\end{lemma}
\begin{proof}
Define a subset $U^\ast=U^\ast(S,T)\subseteq\rho/d$ recursively as follows. As the initial condition, let $(r,\rho_r)\in U^\ast$. If there exists an
$(i,j)\notin U^\ast$ adjacent to an element of $U^\ast$, i.e., $(i\pm 1, j)\in U^\ast$ or $(i, j\pm 1)\in U^\ast$, such that
$S_{i,j} < T_{i,j}$,
then add $(i,j)$ to $U^\ast$, and continue until there is no such $(i,j)$. The process clearly terminates by finiteness.
See Example~\ref{ex:U*NotInject} for an illustration of the set $U^\ast$.

Observe that $U^\ast\subseteq \rho/e$,
because for all $j$ such that $d < j \leq e$, we have $S_{1,j} > 0 = T_{1,j}$, see Notation~\ref{not:defS_U}, so that $(1,j)\notin U^\ast$.
We claim further that $(S_{U^\ast}^\prime, T_{U^\ast}^\prime)\in\bs{S}_{(e,d)}$, i.e., $S_{U^\ast}^\prime$ and $T_{U^\ast}^\prime$ are skew SSYT's.
To see this, we check that each column is increasing, and each row is non-decreasing, at the boundary of $U^\ast$.

Suppose $(i,j)\notin U^\ast$ and $(i+1,j)\in U^\ast$. The first condition implies that
$T_{i,j}\leq S_{i,j}$, while the second implies
$S_{i+1,j} < T_{i+1,j}$. This along with the adjacency of these two positions implies that
$S_{i,j} < S_{i+1,j} < T_{i+1,j}$, and
$T_{i,j} \leq S_{i,j} < S_{i+1,j}$,
so that the $j$th column of both $S_{U^\ast}^\prime$ and $T_{U^\ast}^\prime$ is increasing from row $i$ to row $i+1$.

Suppose $(i,j)\in U^\ast$ and $(i+1,j)\notin U^\ast$. It follows that
$S_{i,j} < T_{i,j}$ and
$T_{i+1,j} \leq S_{i+1,j}$. Then
$T_{i,j} < T_{i+1,j} \leq S_{i+1,j}$, and
$S_{i,j} < T_{i,j} < T_{i+1,j}$
so that both $S_{U^\ast}^\prime$ and $T_{U^\ast}^\prime$ increase from
$(i,j)$ to $(i+1,j)$.

If $(i,j)\notin U^\ast$ and $(i,j+1)\in U^\ast$, then
$T_{i,j} \leq S_{i,j}$ and
$S_{i,j+1} < T_{i,j+1}$. Hence
$S_{i,j} \leq S_{i,j+1} < T_{i,j+1}$, and
$T_{i,j} \leq S_{i,j} \leq S_{i,j+1}$ so that $S_{U^\ast}^\prime$ and $T_{U^\ast}^\prime$ are non-decreasing from
$(i,j)$ to $(i,j+1)$.

If $(i,j)\in U^\ast$ and $(i,j+1)\notin U^\ast$, then
$S_{i,j} < T_{i,j}$ and
$T_{i,j+1} \leq S_{i,j+1}$. Hence
$T_{i,j} \leq T_{i,j+1} \leq S_{i,j+1}$,
$S_{i,j} < T_{i,j} \leq T_{i,j+1}$,
and $S_{U^\ast}^\prime$ and $T_{U^\ast}^\prime$ are non-decreasing from $(i,j)$ to $(i,j+1)$.

It follows that at every position on the boundary of $U^\ast$, the skew SSYT requirements are satisfied so that
$(S_{U^\ast}^\prime, T_{U^\ast}^\prime)\in\bs{S}_{(e,d)}$.
By Lemma~\ref{lem:UMinimal}, we then have that $U(S,T)\subseteq U^\ast\subseteq\rho/e$, completing the proof.
\end{proof}

Given the results of Lemmas~\ref{lem:UIndepSpanTree} and \ref{lem:UinT}, we now have a well-defined map
$\varphi\colon\bs{S}_{(d,e)}\to\bs{S}_{(e,d)}$ given by $\varphi(S,T) = (S_{U(S,T)}^\prime, T_{U(S,T)}^\prime)$.
At this point, the reader may wonder why we do not use the set $U^\ast(S,T)$ as defined constructively in the proof of
Lemma~\ref{lem:UinT} in place of $U(S,T)$. While this would as well yield a well-defined map $\bs{S}_{(d,e)}\to\bs{S}_{(e,d)}$,
the map would not be injective as illustrated by the following.

\begin{example}
\label{ex:U*NotInject}
Let $\rho = (5,4)$, $\lambda = (4,4)$, $d = 1$, and $e = 2$, and let $(S_1,T_1), (S_2, T_2)\in\bs{S}_{(d,e)}$ with $S_i$ and $T_i$ the
skew SSYT's pictured in Figure~\ref{fig:ExNotInj}. The sets $U^\ast(S_i,T_i)$, $i=1,2$, described in the proof of Lemma~\ref{lem:UinT} are shaded in each pair.
The pair
\[
    \big((S_1)_{U^\ast(S_1,T_1)}^\prime,(T_1)_{U^\ast(S_1,T_1)}^\prime\big) = \big((S_2)_{U^\ast(S_2,T_2)}^\prime,(T_2)_{U^\ast(S_2,T_2)}^\prime\big)
\]
is pictured as well, indicating that the assignment
$(S,T)\mapsto(S_{U^\ast(S,T)}^\prime,T_{U^\ast(S,T)}^\prime)$, using the definition of $U^\ast(S, T)$ in Lemma~\ref{lem:UinT}, is not injective.

\begin{figure}[!htb]
\centering
\begin{subfigure}{.4\textwidth}
\centering
\ytableausetup{mathmode, boxframe=normal, boxsize=1.7em}
\begin{ytableau}
\none &  2   &   *(lightgray)  2  &   *(lightgray)  3  &   *(lightgray)  3
\\
 2  &   3   &    4  &    5
\end{ytableau}
\caption*{$S_1$}
\end{subfigure}
\begin{subfigure}{.4\textwidth}
\centering
\ytableausetup{mathmode, boxframe=normal, boxsize=1.7em}
\begin{ytableau}
\none & \none &   3  &  4
\\
 2  &    3   &    4  &  5
\end{ytableau}
\caption*{$T_1$}
\end{subfigure}

\bigskip\bigskip

\begin{subfigure}{.4\textwidth}
\centering
\ytableausetup{mathmode, boxframe=normal, boxsize=1.7em}
\begin{ytableau}
\none &  2   &   3  &   *(lightgray)  3  &   *(lightgray)  3
\\
 2  &   3   &    4  &    5
\end{ytableau}
\caption*{$S_2$}
\end{subfigure}
\begin{subfigure}{.4\textwidth}
\centering
\ytableausetup{mathmode, boxframe=normal, boxsize=1.7em}
\begin{ytableau}
\none & \none &   2  &  4
\\
 2  &    3   &    4  &  5
\end{ytableau}
\caption*{$T_2$}
\end{subfigure}

\bigskip\bigskip

\begin{subfigure}{.4\textwidth}
\centering
\ytableausetup{mathmode, boxframe=normal, boxsize=1.7em}
\begin{ytableau}
\none &  \none   &  2  &  3  &  3
\\
 2  &   3   &    4  &    5
\end{ytableau}
\caption*{$(S_1)_{U^\ast(S_1,T_1)}^\prime = (S_2)_{U^\ast(S_2,T_2)}^\prime$}
\end{subfigure}
\begin{subfigure}{.4\textwidth}
\centering
\ytableausetup{mathmode, boxframe=normal, boxsize=1.7em}
\begin{ytableau}
\none &  2   &    3  &  4
\\
 2  &    3   &    4  &  5
\end{ytableau}
\caption*{$(T_1)_{U^\ast(S_1,T_1)}^\prime = (T_2)_{U^\ast(S_2,T_2)}^\prime$}
\end{subfigure}
\caption{The skew SSYT pairs $(S_1, T_1)$ and $(S_2, T_2)$ as well as the pair
$\big((S_1)_{U^\ast(S_1,T_1)}^\prime,(T_1)_{U^\ast(S_1,T_1)}^\prime\big) = \big((S_2)_{U^\ast(S_2,T_2)}^\prime,(T_2)_{U^\ast(S_2,T_2)}^\prime\big)$
discussed in Example~\ref{ex:U*NotInject}.}
\label{fig:ExNotInj}
\end{figure}
\end{example}

Using the set $U(S,T)$ defined in terms of a spanning tree, the situation in Example~\ref{ex:U*NotInject} does not occur, and the map $\varphi$
is injective, which we now demonstrate.

\begin{proposition}
\label{prop:Inject}
The function $\varphi\colon\bs{S}_{(d,e)}\to\bs{S}_{(e,d)}$ given by $\varphi(S,T) = (S_{U(S,T)}^\prime, T_{U(S,T)}^\prime)$
is injective.
\end{proposition}
\begin{proof}
Let $(S, T), (S^\ast, T^\ast)\in\bs{S}_{(d,e)}$ such that $\varphi(S,T) = \varphi(S^\ast, T^\ast)$.
Let $U = U(S, T)$ and $U^\ast = U(S^\ast, T^\ast)$, and let $\{(B_k,C_k)\}_{1\leq k\leq M}$ and $\{(B_k^\ast,C_k^\ast)\}_{1\leq k\leq M^\ast}$
be choices of spanning trees for $(S, T)$ and $(S^\ast, T^\ast)$, respectively. Assume for contradiction that $U\nsubseteq U^\ast$,
let $k$ be the smallest value such that $B_k\notin U^\ast$, and then $C_k\in U^\ast$. Let $B_k = (i,j)$ and observe the following.
\begin{enumerate}
\item[\textit{i.}]
As
$\varphi(S,T) = \varphi(S^\ast, T^\ast)$ and $B_k = (i,j)\in U\smallsetminus U^\ast$, we have
$S_{i,j} = T_{i,j}^\ast$ and
$T_{i,j} = S_{i,j}^\ast$.
\end{enumerate}
There are four cases to consider for $C_k$ with similar arguments for each.

For the first case, assume $C_k = (i+1,j)$.
\begin{enumerate}
\item[\textit{ii.}]
As
$\varphi(S,T) = \varphi(S^\ast, T^\ast)$ and $C_k = (i+1,j)\in U\cap U^\ast$, we have
$S_{i+1,j} = S_{i+1,j}^\ast$ and
$T_{i+1,j} = T_{i+1,j}^\ast$.
\item[\textit{iii.}]
As $(S^\ast, T^\ast)\in\bs{S}_{(d,e)}$, we have
$S_{i+1,j}^\ast > S_{i,j}^\ast$ and
$T_{i+1,j}^\ast > T_{i,j}^\ast$.
\end{enumerate}
Combining \textit{i}, \textit{ii.}, and \textit{iii.}, it follows that
$S_{i+1,j} = S_{i+1,j}^\ast > S_{i,j}^\ast = T_{i,j}$ and
$T_{i+1,j} = T_{i+1,j}^\ast > T_{i,j}^\ast = S_{i,j}$,
contradicting the fact that $(B_k, C_k)$ provides a pair of adjacent positions at which one or both of
$S_{U_{k-1}}^\prime$ and $T_{U_{k-1}}^\prime$ fail the skew SSYT requirements.

For the second case, assume $C_k = (i,j+1)$.
\begin{enumerate}
\item[\textit{ii.$^\prime$}]
As
$\varphi(S,T) = \varphi(S^\ast, T^\ast)$ and $C_k = (i,j+1)\in U\cap U^\ast$, we have
$S_{i,j+1} = S_{i,j+1}^\ast$ and
$T_{i,j+1} = T_{i,j+1}^\ast$.
\item[\textit{iii.$^\prime$}]
As $(S^\ast, T^\ast)\in\bs{S}_{(d,e)}$, we have
$S_{i,j+1}^\ast \geq S_{i,j}^\ast$ and
$T_{i,j+1}^\ast \geq T_{i,j}^\ast$.
\end{enumerate}
Combining \textit{i.}, \textit{ii.$^\prime$}, and \textit{iii.$^\prime$}, we have
$S_{i,j+1} = S_{i,j+1}^\ast\geq S_{i,j}^\ast = T_{i,j}$
and
$T_{i,j+1} = T_{i,j+1}^\ast\geq T_{i,j}^\ast = S_{i,j}$,
contradicting that $(B_k, C_k)$ is a pair of adjacent positions at which one or both of
$S_{U_{k-1}}^\prime$ and $T_{U_{k-1}}^\prime$ fail the skew SSYT requirements.

For the third case, assume $C_k = (i-1,j)$.
\begin{enumerate}
\item[\textit{ii.$^{\prime\prime}$}]
As
$\varphi(S,T) = \varphi(S^\ast, T^\ast)$ and $C_k = (i-1,j)\in U\cap U^\ast$, we have
$S_{i-1,j} = S_{i-1,j}^\ast$ and
$T_{i-1,j} = T_{i-1,j}^\ast$.
\item[\textit{iii.$^{\prime\prime}$}]
As $(S^\ast, T^\ast)\in\bs{S}_{(d,e)}$, we have
$S_{i-1,j}^\ast < S_{i,j}^\ast$ and
$T_{i-1,j}^\ast < T_{i,j}^\ast$.
\end{enumerate}
Combining \textit{i.}, \textit{ii.$^{\prime\prime}$}, and \textit{iii.$^{\prime\prime}$}, we have
$S_{i-1,j} = S_{i-1,j}^\ast < S_{i,j}^\ast = T_{i,j}$ and
$T_{i-1,j} = T_{i-1,j}^\ast < T_{i,j}^\ast = S_{i,j}$,
contradicting the failure of the skew SSYT requirements for
$S_{U_{k-1}}^\prime$ and $T_{U_{k-1}}^\prime$ at positions $(i-1,j)$ and $(i,j)$.

For the final case, assume $C_k = (i,j-1)$.
\begin{enumerate}
\item[\textit{ii.$^{\prime\prime\prime}$}]
As
$\varphi(S,T) = \varphi(S^\ast, T^\ast)$ and $C_k = (i,j-1)\in U\cap U^\ast$, we have
$S_{i,j-1} = S_{i,j-1}^\ast$ and
$T_{i,j-1} = T_{i,j-1}^\ast$.
\item[\textit{iii.$^{\prime\prime\prime}$}]
As $(S^\ast, T^\ast)\in\bs{S}_{(d,e)}$, we have
$S_{i,j-1}^\ast \leq S_{i,j}^\ast$ and
$T_{i,j-1}^\ast \leq T_{i,j}^\ast$.
\end{enumerate}
Combining \textit{i.}, \textit{ii.$^{\prime\prime\prime}$}, and \textit{iii.$^{\prime\prime\prime}$},
$S_{i,j-1} = S_{i,j-1}^\ast \leq S_{i,j}^\ast = T_{i,j}$ and
$T_{i,j-1} = T_{i,j-1}^\ast \leq T_{i,j}^\ast = S_{i,j}$,
contradicting the failure of the skew SSYT requirements for
$S_{U_{k-1}}^\prime$ and $T_{U_{k-1}}^\prime$ at $(i,j-1)$ and $(i,j)$.

In each case, we have a contradiction, implying that $U\subseteq U^\ast$, and mutatis mutandis, $U^\ast\subseteq U$,
so that $U = U^\ast$. Along with $\varphi(S,T) = \varphi(S^\ast, T^\ast)$, this implies that
$(S, T) = (S^\ast, T^\ast)$, completing the proof.
\end{proof}

\begin{remark}
\label{rem:PhiSimpleCase}
By Lemma~\ref{lem:ProofSimpleCase}, when $N = n+1$ and no $\lambda_i$ nor $\rho_i$ are zero, $\bs{S}_{(0,1)}$ and $\bs{S}_{(1,0)}$ are in
bijective correspondence. As $\varphi$ is injective, it follows that $\varphi\colon\bs{S}_{(0,1)}\to\bs{S}_{(1,0)}$ is a bijection in this case.
\end{remark}

By Proposition~\ref{prop:Inject}, it follows that whenever $e \neq d$, $\bs{P}(d,e) + \bs{P}(e,d)$ consists of nonpositive terms;
see Notation~\ref{not:S(d,e)P(d,e)}. Hence, each positive term in the numerator of the right side of Equation~\eqref{eq:PartDerivAsSkew},
i.e., terms such that $e > d$, is cancelled by a negative term with the values of $e$ and $d$ exchanged. We therefore have the following.

\begin{theorem}
\label{thrm:PartDerDecOneBox}
Let $\bs{x} = (x_1,x_2,\ldots,x_N)$ be a set of variables, and let $\lambda = (\lambda_1,\lambda_2,\ldots,\lambda_n)$
and $\rho = (\rho_1,\rho_2,\ldots,\rho_n)$ be partitions such that
$\lambda$ is given by $\rho$ with a single box removed, i.e., there is an $r$ such that $\lambda_i = \rho_i$
for $i\neq r$ and $\lambda_r = \rho_r - 1$.
Then for positive values of the $x_i$, $s_{\lambda}(\bs{x})/s_{\rho}(\bs{x})$ is strictly decreasing in
each variable. More strongly, for each $j$, the partial derivative
\begin{equation}
\label{eq:PartDerivOneBox}
    \frac{\partial}{\partial x_j}\left(\frac{s_{\lambda}(\bs{x})}{s_{\rho}(\bs{x})}\right)
\end{equation}
can be expressed as the ratio of polynomials where the coefficients of the denominator are nonnegative and the coefficients of the
numerator are nonpositive.
\end{theorem}


\subsection{The general case}
\label{subsec:ProofGeneralCase}

We now drop the assumption that $\lambda$ and $\rho$ differ by a single box and return to the general case where
$\lambda = (\lambda_1,\lambda_2,\ldots,\lambda_n)$ and $\rho = (\rho_1,\rho_2,\ldots,\rho_n)$ are partitions such
that $\lambda_i \leq \rho_i$ for each $i$ and $\lambda_i < \rho_i$ for at least one $i$. The proof
of Theorem~\ref{mainthm:PartDerDec} is a quick application of Theorem~\ref{thrm:PartDerDecOneBox}.

\begin{proof}[Proof of Theorem~\ref{mainthm:PartDerDec}]
Let $\{ \mu^i \}_{i=0}^r$ be a sequence of partitions such that $\mu^0 = \rho$, $\mu^r = \lambda$, and $\mu^{i+1}$ is formed from
$\mu^{i}$ by removing a single box, i.e., $\mu^{i+1}$ and $\mu^{i}$ satisfy the hypotheses of Theorem~\ref{thrm:PartDerDecOneBox},
for $i=0,1,\ldots,r-1$. Then
\[
    \frac{s_{\lambda}(\bs{x})}{s_{\rho}(\bs{x})}
    =
    \prod\limits_{i=0}^{r-1} \frac{s_{\mu^{i+1}}(\bs{x})}{s_{\mu^i}(\bs{x})}
\]
so that for any $j = 1, 2, \ldots, N$,
\[
    \frac{\partial}{\partial x_j} \left( \prod\limits_{i=0}^{r-1} \frac{s_{\mu^{i+1}}(\bs{x})}{s_{\mu^i}(\bs{x})} \right)
    =
    \sum\limits_{i=0}^{r-1}
        \prod\limits_{\substack{k=0 \\ k\neq i}}^{r-1} \frac{s_{\mu^{k+1}}(\bs{x})}{s_{\mu^k}(\bs{x})}
        \frac{\partial}{\partial x_j} \left(\frac{s_{\mu^{i+1}}(\bs{x})}{s_{\mu^i}(\bs{x})}\right).
\]
By Theorem~\ref{thrm:PartDerDecOneBox}, each $\frac{\partial}{\partial x_j} \left(\frac{s_{\mu^{i+1}}(\bs{x})}{s_{\mu^i}(\bs{x})}\right)$
can be expressed as the ratio of polynomials where the coefficients of the denominator are nonnegative and the coefficients of the
numerator are nonpositive. As $s_{\mu^k}(\bs{x})$ and $s_{\mu^{k+1}}(\bs{x})$ are sums of positive terms, the result follows.
\end{proof}


\section{An application to the Hilbert series of linear symplectic quotients}
\label{sec:SympQuot}

In this section, we discuss the question of distinguishing between linear symplectic quotients using the first Laurent coefficient of the Hilbert series,
including an application of Theorem~\ref{mainthm:PartDerDec}.

If $G$ is a compact Lie group and $V$ is a unitary representation of $G$ considered as the underlying real symplectic vector space, then there is a
\emph{moment map} $\mu\colon V\to \mathfrak{g}^\ast$ where $\mathfrak{g}^\ast$ is the dual of the Lie algebra $\mathfrak{g}$ of $G$, which can
be taken to be homogeneous quadratic. The \emph{shell} $Z = \mu^{-1}(0)$ is a $G$-invariant real algebraic variety in $V$ that is usually singular,
and the \emph{real linear symplectic quotient at level $0$} is the space $M_0$ of $G$-orbits in $Z$, i.e., $M_0 = Z/G$, with the quotient topology.
There is a graded Poisson algebra $\R[M_0]$ of \emph{real
regular functions} on $M_0$ given by the quotient $\R[V]^G/I_Z^G$ where $\R[V]^G$ denotes the real $G$-invariant polynomials on $V$,
$I_Z$ is the ideal of functions that vanish on $Z$, and $I_Z^G = I_Z\cap\R[V]^G$.
We refer the reader to \cite{ArmsGotayJennings,SjamaarLerman} for background on singular symplectic quotients.
Important equivalence relations between linear symplectic quotients are given by \emph{graded regular diffeomorphisms} and the stronger
\emph{graded regular symplectomorphisms}. For the present purposes, it is sufficient to note that a graded regular diffeomorphism
(respectively, symplectomorphism) induces an isomorphism between the graded (respectively, graded Poisson) algebras of real regular functions;
see \cite[Section~4]{FarHerSea} or \cite[Section~2]{HerbigSchwarzSeaton} for more details.

Because $\R[M_0]$ is a finitely generated graded algebra, its \emph{Hilbert series} $\Hilb(t)$, the generating function of the dimensions of the vector spaces
$\R[M_0]_d$ of homogeneous polynomials of degree $d$, is a rational function with a pole at $t = 1$. The coefficients of the Laurent expansion of $\Hilb(t)$
at $t = 1$ are denoted $\gamma_i$ so that
\[
    \Hilb(t)    =   \gamma_0(1 - t)^{-D} + \gamma_1(1 - t)^{1 - D} + \gamma_2(1 - t)^{2 - D} + \cdots
\]
Graded regular diffeomorphisms preserve the coefficients $\gamma_i$ so that the $\gamma_i$ can be used to distinguish between
non-graded regularly diffeomorphic symplectic quotients.
Note that it was recently demonstrated that among symplectic quotients by groups $G$ whose identity component is a torus,
the existence of a regular diffeomorphism implies the existence of a graded regular diffeomorphism; see \cite[Corollary~8.8]{HerbigSchwarzSeatonTorus}.

\begin{remark}
\label{rem:ComplexSympQuot}
There is a related notion of a \emph{complex linear symplectic quotient}, a complex algebraic variety associated to the induced representation of
the complexification of $V$. For almost all representations (and for all representations, using a modified definition of the complex symplectic quotient),
the Hilbert series of the algebra of real regular functions coincides with the Hilbert series of the algebra of complex regular functions;
see \cite[Section~2.2]{HerbigSchwarzSeaton2}. Hence, the $\gamma_i$ can as well be used to distinguish between complex symplectic quotients up to graded
isomorphism.
\end{remark}

A formula for the first few coefficients $\gamma_i$ of the Laurent expansion of the Hilbert series was computed in \cite{HerbigSeatonHSeries}
in the case that $G = \Sp^1$ is the circle; let us briefly recall the formula for $\gamma_0$.
An $n$-dimensional unitary representation of $\Sp^1$ is determined by a \emph{weight vector} $A = (a_1,\ldots,a_n)\in\Z^n$.
As zero weights do not change the value of $\gamma_0$, we assume with no loss of generality that each $a_i\neq 0$.
We assume $\gcd(a_1,\ldots,a_n) = 1$, i.e., that the representation is faithful, for otherwise, we can mod out by the kernel
of the action without changing the symplectic quotient.
We further assume that $A$ contains both positive and negative weights, for otherwise, the symplectic quotient is a point,
With these assumptions, the (real) dimension of the corresponding symplectic quotient is $2n - 2$. Let
$\alpha_i = |a_i|$ denote the absolute value of the $i$th weight and $\bs{\alpha} = (\alpha_1,\ldots,\alpha_n)$, and then by
\cite[Theorem~5.1]{HerbigSeatonHSeries},
\begin{equation}
\label{eq:Gam0S1}
    \gamma_0^{\Sp^1}(A)
    =
    \frac{s_{(n-2,n-2,n-3,\ldots,1,0)}(\bs{\alpha})}
        {s_{(n-1,n-2,n-3,\ldots,1,0)}(\bs{\alpha})}.
\end{equation}

Because the formula for $\gamma_0^{\Sp^1}(A)$ in Equation~\eqref{eq:Gam0S1} satisfies the hypotheses of Theorem~\ref{mainthm:PartDerDec}
(and even Theorem~\ref{thrm:PartDerDecOneBox}), we have the following consequence.

\begin{corollary}
\label{cor:S1}
Let $A = (a_1,\ldots,a_n)$ and $B = (b_1,\ldots,b_n)\in \Z^n$ be weight vectors for $n$-dimensional unitary representations of $\Sp^1$
such that $0 < |a_i| \leq |b_i|$ for each $i$ and $|a_i| < |b_i|$ for at least one $i$.
Assume $\gcd(a_1,\ldots,a_n) = \gcd(b_1,\ldots,b_n) = 1$ and $A$ and $B$ contain both positive and negative weights.
Then $\gamma_0^{\Sp^1}(A) > \gamma_0^{\Sp^1}(B)$. In particular, the corresponding
symplectic quotients are not graded regularly diffeomorphic.
\end{corollary}

It follows that many examples of symplectic quotients associated to circles are not graded regularly symplectomorphic. This fact was already
established more generally by \cite[Theorem~H]{HerbigSchwarzSeatonTorus}; however, Corollary~\ref{cor:S1} indicates that the invariant $\gamma_0^{\Sp^1}(A)$
is sufficiently fine to distinguish between many examples.

We also have the following bound for $\gamma_0^{\Sp^1}(A)$ in terms of the largest two weights.

\begin{proposition}
\label{prop:S1Bound}
Let $A = (a_1,\ldots,a_n)\in\Z^n$ be the weight vector of an $n$-dimensional unitary representation of $\Sp^1$.
Assume that each $a_i\neq 0$, $\gcd(a_1,\ldots,a_n) = 1$, and $A$ contains both positive and negative weights.
Let $i_0, i_1\in\{1,2,\ldots,n\}$ with $i_0\neq i_1$ such that $\alpha_{i_0}$ and $\alpha_{i_1}$ are the largest two
$\alpha_i$ in $\bs{\alpha}$ (which may be equal). Then
\begin{equation}
\label{eq:S1Bound}
    \gamma_0^{\Sp^1}(A) \leq    \frac{1}{\alpha_{i_0} + \alpha_{i_1}}.
\end{equation}
\end{proposition}
\begin{proof}
Let $\lambda = (n-2,n-2,n-3,\ldots,1,0)$ and $\delta = (n-1,n-2,n-3,\ldots,1,0)(\bs{\alpha})$.
If $S$ is a SSYT of shape $\lambda$ with labels in $\{1,2,\ldots,n\}$, then as the last column must be increasing, the largest possible value of
$S_{1,n-2}$ is $n-1$. Therefore, there are at least two SSYT's of shape $\delta$ that result in $S$ when removing the box in position $(1,n-1)$;
the first has label $n-1$ at position $(1,n-1)$ while the second has label $n$ in this position. It follows that if $\bs{x} = (x_1,\ldots,x_n)$
and all $x_i\geq 0$, we have
\[
    (x_n + x_{n-1})s_\lambda(\bs{x})
    \leq
    s_\delta(\bs{x}).
\]
Then
\[
    \frac{s_\lambda(\bs{x})}{s_\delta(\bs{x})}
    \leq
    \frac{1}{x_n + x_{n-1}}
\]
so that as $s_\lambda(\bs{x})$ and $s_\delta(\bs{x})$ are symmetric,
\[
    \frac{s_\lambda(\bs{x})}{s_\delta(\bs{x})}
    \leq
    \frac{1}{x_i + x_j}
\]
for each $i\neq j$. Substituting $\bs{\alpha}$ for $\bs{x}$ completes the proof.
\end{proof}

Proposition~\ref{prop:S1Bound} can be used, for instance, to distinguish between symplectic quotients by $\Sp^1$ with those by other groups.
For instance, in the case that $G = \SU_2$ is the special unitary $2\times 2$ group, there are finitely many unitary representations of a given
dimension, and hence finitely many $\SU_2$-symplectic quotients of a given dimension. We therefore have the following.

\begin{corollary}
\label{cor:Sp1SU2}
For any dimension $m$, all but finitely many representations of $\Sp^1$ with symplectic quotients of dimension $m$
have values of $\gamma_0$ that do not coincide with that of any symplectic quotient by $\SU_2$ of dimension $m$. Hence, for any fixed dimension
$m$, all but finitely many symplectic quotients by $\Sp^1$ are not graded regularly diffeomorphic to a symplectic quotient by
$\SU_2$.

In particular, let $\gamma_0^{\SU_2,\operatorname{min}}(m)$ denote the minimum value of $\gamma_0$ among the finitely many representations
of $\SU_2$ with symplectic quotients of dimension $m$. If the symplectic quotient corresponding to the $\Sp^1$-representation with
weight vector $A = (a_1,\ldots,a_n)$ is graded regularly diffeomorphic to a symplectic quotient by $\SU_2$, then the largest two elements
$\alpha_{i_0}$ and $\alpha_{i_1}$ of $\bs{\alpha}$ satisfy $\alpha_{i_0} + \alpha_{i_1} \leq 1/\gamma_0^{\SU_2,\operatorname{min}}(m)$.
\end{corollary}

It follows that for any symplectic quotient dimension $m$, there are finitely many weight vectors $A$ for which it is possible that
$\gamma_0^{\Sp^1}(A)$ coincides with $\gamma_0$ for an $\SU_2$-symplectic quotient of the same dimension.

Let us recall the computation of $\gamma_0$ in the case $G = \SU_2$ from \cite{HerbigHerdenSeatonSU2}.
Let $V_d$ denote the unique unitary irreducible representation of $\SU_2$ of complex dimension $d+1$ on binary forms of degree $d$;
the weights of $V_d$ are $-d, -d+2, \ldots, d-2, d$.
Then a unitary representation $V$ of $\SU_2$ such that $V^{\SU_2} = \{0\}$ is of the form $V = \bigoplus_{k=1}^r V_{d_k}$ for some $d_k \geq 1$,
and the weights of $V$ are given by combining
the weights for each $V_{d_k}$ accounting for multiplicities. Fix such a representation $V$ and let $\bs{d} = (d_1,\ldots,d_r)$.
Let $\sigma_V = 2$ if each $d_k$ is even and $1$ otherwise, and let $\bs{a}$ denote the collection of \emph{positive} weights of the representation
$V\oplus V^\ast\simeq V\oplus V$, i.e., the positive weights of $V$ with each multiplicity doubled. Let $D = \dim_\C V = r + \sum_{k=1}^r d_k$ denote the
complex dimension of $V$, let $e$ denote the number of even $d_k$, and let $C = (D - e)/2$ denote the number of positive weights in $V$
so that $\bs{a} = (a_1,a_2,\ldots,a_{2C})$. Let $\widehat{\delta} = (2C - 1, 2C - 2, \ldots, 1, 0)$,
$\widehat{\rho} = (2C - 3, 2C - 3, 2C - 3, 2C - 4,\ldots, 1, 0)$, and
$\widehat{\rho}^\prime = (2C - 3, 2C - 4, 2C - 4, 2C - 4,2C - 5,\ldots, 1, 0)$.
Then by \cite[Theorem~4.2]{HerbigHerdenSeatonSU2}, assuming $V$ is not isomorphic to $V_d$ for $d = 1,2,3,4$ nor $2V_1$, the symplectic quotient
has real dimension $2D - 6 = 4C + 2e - 6$, and $\gamma_0$ is given by
\begin{equation}
\label{eq:Gam0SU2}
    \gamma_0^{\SU_2}(\bs{d})
    =
    \frac{8\sigma_V \big(s_{\widehat{\rho}}(\bs{a})
        + s_{\widehat{\rho}^\prime}(\bs{a})\big)}
        {s_{\widehat{\delta}}(\bs{a})}.
\end{equation}

This expression can be used to find a lower bound for $\gamma_0^{\SU_2}(\bs{d})$ among representations $V$ of dimension $D$, i.e., symplectic quotients
of dimension $2D-6$, as follows.
Given a SSYT $S$ of shape $\widehat{\rho}$, the possible SSYT's $T$ of shape $\widehat{\delta}$ that yield $S$ when the boxes in positions
$(1,2C-2)$, $(1,2C-1)$ and $(2,2C-2)$ are removed provide monomials of the form $\bs{x}^T = \bs{x}^S x_i x_j x_k$ where $j\geq i\geq 1$ and $k > i$. Hence,
\[
    s_{\widehat{\rho}}(\bs{a})\sum\limits_{\substack{1\leq i\leq j \leq 2C \\ i < k \leq 2C}} a_i a_j a_k
    \geq
    s_{\widehat{\delta}}(\bs{a}).
\]
Therefore,
\[
    \gamma_0^{\SU_2}(\bs{d})
    >
    \frac{8\sigma_V s_{\widehat{\rho}}(\bs{a})}
        {s_{\widehat{\delta}}(\bs{a})}
    \geq
    \frac{8\sigma_V}{\sum\limits_{\substack{1\leq i\leq j \leq 2C \\ i < k \leq 2C}} a_i a_j a_k}.
\]
By induction on $C$, one can show that the sum in the denominator contains $2(4C^3-C)/3$ terms,
and as $2C\leq D$, we have $2(4C^3-C)/3 \leq (D^3 - D)/3$.
Each $d_k \leq D-1$, and each $a_i, a_j$, and $a_k$ is bounded by the maximum value of $d_k$. Therefore, $a_j, a_k \leq D-1$ so that
$a_i \leq D-2$. Hence,
\[
    \gamma_0^{\SU_2}(\bs{d})
    >
    \frac{8\sigma_V}{\frac{1}{3}(D^3-D) (D-2)(D-1)^2}
    =
    \frac{24\sigma_V}{D(D-2)(D+1)(D-1)^3}.
\]

Let us describe an application of Corollary~\ref{cor:Sp1SU2}.
The $\gamma_0^{\SU_2}(\bs{d})$ have been computed for each
linear $\SU_2$-symplectic quotient of dimension at most $38$ in \cite{HerbigHerdenSeatonSU2} using Equation~\eqref{eq:Gam0SU2}
as well as direct computations of the Hilbert series for the low-dimensional exceptions.
We have compared the values of $\gamma_0^{\Sp^1}(A)$ for each $A$ satisfying the conclusion
$\alpha_{i_0} + \alpha_{i_1} \leq 1/\gamma_0^{\SU_2,\operatorname{min}}(m)$
of Corollary~\ref{cor:Sp1SU2}
to the finite list of values of $\gamma_0^{\SU_2}(\bs{d})$ for symplectic quotients of dimension at most 8.
Recall that the real dimension of an $\Sp^1$-symplectic quotient is $2n - 2$, and the real dimension of an $\SU_2$-symplectic quotient,
excluding some low-dimensional exceptions, is $2D - 6$.
This comparison is vastly simplified using Theorem~\ref{mainthm:PartDerDec}; once a value of $\gamma_0^{\Sp^1}(A)$ is found that is smaller
than $\gamma_0^{\SU_2,\operatorname{min}}(\bs{d})$ of a given symplectic quotient dimension, no other weight vectors with all weights
larger than or equal to those of $A$ must be checked.
The only cases where $\gamma_0^{\Sp^1}(A) = \gamma_0^{\SU_2}(\bs{d})$ for symplectic quotients of the same dimension are as follows.
\begin{itemize}
\item   $\gamma_0^{\Sp^1}(1,1) = \gamma_0^{\SU_2}(1,1) = \gamma_0^{\SU_2}(2) = 1/2$. These $2$-dimensional symplectic quotients are graded
        regularly symplectomorphic, as they are all graded regularly symplectomorphic to the orbifold $\C/\pm 1$;
        see \cite[Lemma~2.7 and Theorem~2.9]{ArmsGotayJennings}, \cite[Theorem~7]{FarHerSea}, and \cite[Remark~5.4]{HerbigSchwarzSeaton}.
\item   $\gamma_0^{\Sp^1}(1,3) = \gamma_0^{\SU_2}(3) = 1/4$. These $2$-dimensional symplectic quotients are also graded regularly symplectomorphic
        to an orbifold, $\C/\langle i \rangle$; see \cite[Theorem~7]{FarHerSea} and \cite[Proposition~5.5]{HerbigSchwarzSeaton}.
\item   $\gamma_0^{\Sp^1}(1,2,2) = \gamma_0^{\SU_2}(1,2) = 2/9$. The graded algebras of real regular functions on these two symplectic quotients,
        both of dimension $4$, have the same Hilbert series. The Hilbert series for the $\SU_2$-quotient was computed in
        \cite[Table~A.1]{HerbigHerdenSeatonSU2} and is given by
        \[
            \frac{1 + 2t^2 + 2t^3 + 2t^4 + t^6}{(1 - t^2)^2(1 - t^3)^2}.
        \]
        This coincides with the Hilbert series of the $\Sp^1$-quotient, which can be computed using \cite[Theorem~3.1]{HerbigSeatonHSeries}
        and is also easy to compute using an explicit description of the real regular functions.
\item   $\gamma_0^{\Sp^1}(1,1,1,1) = \gamma_0^{\SU_2}(1,1,1) = 5/16$. The Hilbert series of the graded algebras of real regular functions on these
        two symplectic quotients, both of dimension $6$, coincide as well; they were also computed in \cite[Table~A.1]{HerbigHerdenSeatonSU2}
        and \cite[Section~5.3]{HerbigSeatonHSeries} and are given by
        \[
            \frac{1 + 9t^2 + 9t^4 + t^6}{(1 - t^2)^6}.
        \]
\end{itemize}

There are no cases of $8$-dimensional symplectic quotients such that $\gamma_0^{\Sp^1}(A) = \gamma_0^{\SU_2}(\bs{d})$, and though
checking in dimension $10$ is in progress, no cases such that $\gamma_0^{\Sp^1}(A) = \gamma_0^{\SU_2}(\bs{d})$ have been identified.

This experimental investigation of the $\gamma_0$ using Corollary~\ref{cor:Sp1SU2} continues and will be used to identify other
potential graded regular diffeomorphisms or symplectomorphisms between symplectic quotients by $\Sp^1$ and $\SU_2$, if they exist.
The question of whether the examples in the third and fourth items of the above list are indeed graded regularly diffeomorphic or
symplectomorphic will be pursued in a more relevant context.


\bibliographystyle{amsplain}
\providecommand{\bysame}{\leavevmode\hbox to3em{\hrulefill}\thinspace}
\providecommand{\MR}{\relax\ifhmode\unskip\space\fi MR }
\providecommand{\MRhref}[2]{%
  \href{http://www.ams.org/mathscinet-getitem?mr=#1}{#2}
}
\providecommand{\href}[2]{#2}

\end{document}